\documentclass[12pt]{article}
\usepackage{amsmath}
\usepackage{amssymb}
\usepackage{amsthm}
\usepackage{verbatim}
\usepackage{mathrsfs}

\newcommand{\R}[0]{\mathbb R}

\newcommand{\Ds}[0]{\mathcal D}

\newtheorem{Th}{Theorem}[section]
\newtheorem{Lemma}{Lemma}[section]

\newtheorem{Rem}{Remark}[section]

\begin{document}

\title{On the local well-posedness of the 1D Green-Naghdi system over a nonflat bottom}
\author{H. Inci}

\maketitle

\begin{abstract}
	In this paper we consider the 1D Green-Naghdi system over a nonflat bottom. This system describes the evolution of water waves over an uneven bottom in the shallow water regime in terms of the water depth $h$ and the horizontal velocity $u$. Using a Lagrangian formulation of this system on a Sobolev type diffeomorphism group we prove local well-posedness for $(h,u)$ in the Sobolev space $(1+H^s(\R)) \times H^{s+1}(\R),\; s > 1/2$. This improves the local well-posedness range.
\end{abstract}

\section{Introduction}\label{section_introduction}

We consider for $t \geq 0,\; x \in \R$, the 1D Green-Naghdi system in the setting of a nonflat bottom as given in \cite{seabra}
\begin{align}
	\nonumber
	&h_t+(hu)_x=0,\\
	\nonumber
	&h(u_t+uu_x)+[\frac{1}{2}h^2g+\frac{1}{2}h^2 \xi_x(u_t+uu_x)+\frac{1}{2}h^2 u^2 \xi_{xx} + \frac{1}{3} h^3 (u_x^2-uu_{xx}-u_{tx})]_x\\
	\label{gn}
	&=-h\xi_x g-h\xi_x^2 (u_t+uu_x)-hu^2 \xi_x \xi_{xx}-\frac{1}{2} h^2 \xi_x (u_x^2-uu_x-u_{tx}),\\
	\nonumber
	&h(t=0)=h_0,\;u(t=0)=u_0,
\end{align}
where $h(t,x) \in \R$ is the water depth, $u(t,x) \in \R$ the horizontal velocity of the water wave, $\xi(x) \in \R$ the shape of the bottom and $g$ the gravitational constant. The Green-Naghdi system \eqref{gn} models 1D wave propagation over an uneven bottom in the shallow water regime, i.e when the typical wavelength is much larger than the water depth.\\ 
A 2D version of \eqref{gn} was studied by Green and Naghdi in \cite{gn}, a flat bottom version of \eqref{gn}, i.e. the case $\xi \equiv 0$, appears in \cite{serre} and \cite{gardner}. As mentioned above \eqref{gn} is taken from \cite{seabra}.\\ \\
Let us define for $h(t,x)$ and $\xi(x)$ the operator
\[
	A_{h,\xi}:u \mapsto h(1+\xi_x^2) u + [\frac{1}{2} h^2 \xi_x u]_x- \frac{1}{2} h^2 \xi_x u_x-[\frac{1}{3} h^3 u_x]_x.
\]
With this we have
\begin{align*}
	&A_{h,\xi}(u_t+uu_x)=h(1+\xi_x^2) (u_t+uu_x)+[\frac{1}{2} h^2 \xi_x (u_t+uu_x)]_x\\
	&-\frac{1}{2} h^2 \xi_x (u_{tx}+u_x^2+uu_{xx})-[\frac{1}{3} h^3 (u_t+uu_{xx}+u_x^2)]_x.
\end{align*}
Thus using the second equation in \eqref{gn} we can write
\begin{align*}
	A_{h,\xi}(u_t+uu_x)=-hh_xg-[\frac{1}{2}h^2 u^2 \xi_{xx}]_x-[\frac{2}{3} h^3 u_x^2]_x-h \xi_x g -hu^2 \xi_x \xi_{xx}- h^2 \xi_x u_x^2.
\end{align*}
Supposing that $A_{h,\xi}$ is invertible the second equation in \eqref{gn} is equivalent to
\begin{align}
	\nonumber
	u_t+uu_x&=-A_{h,\xi}^{-1}\left(hh_xg+[\frac{1}{2}h^2 u^2 \xi_{xx}]_x+[\frac{2}{3} h^3 u_x^2]_x+h \xi_x g +hu^2 \xi_x \xi_{xx}+ h^2 \xi_x u_x^2\right)\\
	\label{nonlocal}
	&=-A_{h,\xi}^{-1} P(h,u,\xi).
\end{align}
We will see in the next section that $A_{h,\xi}$ is invertible under suitable assumptions on $h$. The local well-posedness result for \eqref{gn} we want to prove reads as
\begin{Th}\label{th_lwp}
	Let $s > 1/2$, $\xi \in H^{s+3}(\R)$ and $(h_0,u_0) \in (1+H^s(\R)) \times H^{s+1}(\R)$ with $\inf_{x \in \R} h_0(x) > 0$. Then there is $T > 0$ and a unique solution $(h,u)$ to \eqref{gn} on $[0,T]$ of class
	\[
		(h,u) \in C([0,T];(1+H^s(\R)) \times H^{s+1}(\R)) \cap C^1([0,T];(1+H^{s-1}(\R)) \times H^s(\R)),
	\]
with $\inf_{x \in \R} h(t,x) > 0,\; 0 \leq t \leq T$. Moreover the solution depends continuously on the data $(h_0,u_0,\xi)$.
\end{Th}

\begin{Rem}
	In the statement of Theorem \ref{th_lwp} we could take $\xi \in C_b^{\lceil s \rceil+3}(\R)$ as well for the bottom. 
\end{Rem}

In \cite{israwi} it was proved that \eqref{gn} is locally well-posed for $(h,u) \in (1+H^s(\R)) \times H^{s+1}(\R),\; s > 3/2$. So Theorem \ref{th_lwp} improves the local well-posedness range. Recently in \cite{gn_lagrangian} it was shown that the 1D Green-Naghdi system over a flat bottom is locally well-posed for $(h,u) \in (1+H^s(\R)) \times H^s(\R),\; s > 1/2$. This corresponds to $\xi \equiv 0$ in Theorem \ref{th_lwp}.\\
To prove Theorem \ref{th_lwp} we will use the same approach as in \cite{gn_lagrangian}. We will write \eqref{gn} in Lagrangian variables, i.e. in terms the flow map $\varphi$ of $u$. The flow map $\varphi$ is defined by the ODE
\[
	\varphi_t(t,x)=u(t,\varphi(t,x)),\;t \geq 0,\;\varphi(0,x)=x,\; x \in \R.
\]
Note that this defines for each $t$ a diffeomorphism $\varphi(t):=\varphi(t,\cdot):\R \to \R$. We can translate the first equation in \eqref{gn} into Lagrangian coordinates by considering 
\begin{align*}
	\frac{d}{dt} \left(\varphi_x \cdot h \circ \varphi\right)=\varphi_x \cdot (u_x \circ h) \circ \varphi+\varphi_x \cdot (h_t+u h_x) \circ \varphi = 0,
\end{align*}
where the last equality is a consequence of $h_t+(hu)_x=0$. Thus $h$ expressed in Lagrangian variables is
\[
	h=\left(\frac{h_0}{\varphi_x}\right) \circ \varphi^{-1}.
\]
To express the second equation of \eqref{gn}, i.e. \eqref{nonlocal}, in Lagrangian variables consider
\begin{align}
	\nonumber
	&\varphi_{tt}=\frac{d}{dt} u \circ \varphi = (u_t+uu_x) \circ \varphi\\
	\label{F}
	&=-\left(A_{(h_0/\varphi_x)\circ \varphi^{-1},\xi} P\left((h_0/\varphi_x) \circ \varphi^{-1},\varphi_t \circ \varphi^{-1},\xi \right) \right) \circ \varphi=F(\varphi,\varphi_t,h_0,\xi),
\end{align}
where we used the equation for $u_t+uu_x$ in \eqref{nonlocal}, $h=(h_0/\varphi_x) \circ \varphi^{-1}$ and $u=\varphi_t \circ \varphi^{-1}$. Our strategy will be to prove that $F$ is continuously differentiable on a suitable functional space and hence by solving the second order ODE $\varphi_{tt}=F(\varphi,\varphi_t,h_0,\xi)$ we will get local well-posedness of \eqref{gn}.

\section{The operator $A_{h,\xi}$}\label{section_ah}

The goal of this section is to prove that $A_{h,\xi}$ is invertible.

\begin{Lemma}
	Let $s > 1/2$, $\xi \in H^{s+3}(\R)$ and $h \in 1+H^s(\R)$ with $\inf_{x \in \R} h(x) > 0$. Then 
\[
	A_{h,\xi}:H^{s+1}(\R) \to H^{s-1}(\R),\;u \mapsto h(1+\xi_x^2) u + [\frac{1}{2} h^2 \xi_x u]_x- \frac{1}{2} h^2 \xi_x u_x-[\frac{1}{3} h^3 u_x]_x,
\]
is invertible.
\end{Lemma}

\begin{proof}
Consider the symmetric bilinear form on $H^1(\R)$ given by
	\[
		\langle u,v \rangle_{h,\xi} = \int_{\R} h(1+\xi_x^2) u v -\frac{1}{2} h^2 \xi_x u v_x - \frac{1}{2} h^2 \xi_x u_x v + \frac{1}{3} h^3 u_x v_x \;dx.
	\]
	We want to show that $\langle \cdot,\cdot \rangle_{h,\xi}$ is equivalent to the $H^1$ inner product
	\[
		\langle u,v \rangle_{H^1} = \int_\R uv + u_x v_x \;dx,\; u,v \in H^1(\R).
	\]
By the Sobolev imbedding $H^s(\R)$ imbeds into $C_0(\R)$, the continuous functions vanishing at infinity. Thus
	\[
		|\langle u,u\rangle_{h,\xi}| \leq C \langle u,u \rangle_{H^1}
	\]
	for some $C > 0$. For the estimate from below consider
	\[
		\langle u,u \rangle_{h,\xi}=\int_{\R} h(1+\xi_x^2) u^2-h^2 \xi_x u u_x +\frac{1}{3} h^3 u_x^2 \;dx.
	\]
We have
	\[
		|h^2 \xi_x u u_x| = (|h|^{1/2} |\xi_x| |u|) (|h|^{3/2} |u_x|) \leq h \xi_x^2 u^2 + \frac{1}{4} h^3 u_x^2,
	\]
where in the last step we used $ab \leq \frac{\varepsilon}{2} a^2 + \frac{1}{2\varepsilon} b^2$ with $\varepsilon=2$. Thus
	\[
		\langle u,u \rangle_{h,\xi} \geq \int_\R h u^2 + \frac{1}{12} h^3 u_x^2 \;dx.
	\]
Since $\inf_{x \in \R} h(x) > 0$ we get
	\[
		\langle u,u \rangle_{h,\xi} \geq C \langle u,u \rangle_{H^1}
	\]
	for some $C > 0$. Let $f \in H^{s-1}(\R)$. Our goal is to show that
\[
	A_{h,\xi}(u)=f
\]
	has a unique solution $u \in H^{s+1}(\R)$. We assume first $1/2 < s \leq 1$. Since $f \in H^{-1}(\R)$ we get by the Riesz Representation Theorem a unique $u \in H^1(\R)$ satisfying
	\[
		\langle u,\phi \rangle_{h,\xi} = \langle f,\phi \rangle,\;\forall \phi \in C_c^\infty(\R),
	\]
where $\langle f,\cdot \rangle$ is the duality pairing. In other words we have 
	\[
h(1+\xi_x^2) u + [\frac{1}{2} h^2 \xi_x u]_x- \frac{1}{2} h^2 \xi_x u_x-[\frac{1}{3} h^3 u_x]_x=f
	\]
in $H^{-1}(\R)$. Rewriting gives
	\[
		[\frac{1}{3} h^3 u_x]_x=h(1+\xi_x^2) u + [\frac{1}{2} h^2 \xi_x u]_x- \frac{1}{2} h^2 \xi_x u_x-f \in H^{s-1}(\R).
	\]
	Thus $h^3 u_x \in H^s(\R)$. From \cite{composition} we know that dividing by $h^3$ is a bounded linear map on $H^s(\R)$. So $u_x \in H^s(\R)$ and with that we conclude $u \in H^{s+1}(\R)$ and that $u$ satisfies $A_{h,\xi}(u)=f$ in $H^{s-1}(\R)$. Suppose now $1 < s \leq 2$. From the previous step (case $s=1$) we get $u \in H^2(\R)$ satisfying $A_{h,\xi}(u)=f$ in $L^2$. Thus we have
\[
		[\frac{1}{3} h^3 u_x]_x=h(1+\xi_x^2) u + [\frac{1}{2} h^2 \xi_x u]_x- \frac{1}{2} h^2 \xi_x u_x-f \in H^{s-1}(\R).
\]
	Arguing as above we get $u \in H^{s+1}(\R)$ satisfying $A_{h,\xi}(u)=f$ in $H^{s-1}(\R)$. By continuing this bootstrap argument for $2 < s \leq 3,\;3 < s \leq 4, \ldots$, we get that
	\[
		A_{h,\xi}:H^{s+1}(\R) \to H^{s-1}(\R)
	\]
is an isomorphism for $s > 1/2$.
\end{proof}

\section{Lagrangian formulation}\label{section_lagrangian}

The goal of this section is to write \eqref{gn} in Lagrangian variables. Let us start with some basic results about Sobolev spaces -- see \cite{composition} for the proofs. Let $s > 1/2$. Then multiplication
\[
	H^\sigma(\R) \times H^s(\R) \to H^\sigma(\R),\;(f,g) \mapsto f \cdot g,\quad 0\leq \sigma \leq s,
\]
is continuous. For $1/2 < s < 1$ this extends to the range $s-1 \leq \sigma < 0$ -- see \cite{gn_lagrangian}. By the Sobolev Imbedding Theorem we have
\[
	H^s(\R) \hookrightarrow C_0(\R),\;H^{s+1}(\R) \hookrightarrow C_0^1(\R).
\]
In \cite{composition} the authors studied the functional space
\[
	\Ds^{s+1}(\R)=\{\varphi:\R \to \R \;|\; \varphi-\text{id} \in H^{s+1}(\R),\;\varphi_x(x) > 0 \; \forall x \in \R \},
\]
where $\text{id}:\R \to \R,\;x \mapsto x$, is the identity map on $\R$. By the Sobolev Imbedding Theorem $\Ds^{s+1}(\R)$ consists of $C^1$ diffeomorphisms of $\R$ and $\Ds^{s+1}(\R)-\text{id}$ is an open subset of $H^{s+1}(\R)$. So as an open subset of $H^{s+1}(\R)$ it has naturally a differential structure. In \cite{composition} it was shown that the composition map
\begin{align}\label{composition}
	H^\sigma(\R) \times \Ds^{s+1}(\R) \to H^\sigma(\R),\;(f,\varphi) \mapsto f \circ \varphi,\quad 0 \leq \sigma \leq s+1,
\end{align}
is continuous. Furthermore we have that
\begin{align}\label{C1}
	H^{s+1}(\R) \times \Ds^{s+1}(\R) \to H^s(\R),\;(f,\varphi) \mapsto f \circ \varphi,
\end{align}
is a $C^1$ map. In \cite{composition} it was also shown that
\begin{align}\label{inversion}
	\Ds^{s+1}(\R) \to \Ds^{s+1}(\R),\;\varphi \mapsto \varphi^{-1},
\end{align}
is continuous. In particular we get that $(\Ds^{s+1}(\R),\circ)$ is a topological group. If we denote the composition from the right with $\varphi$ by $R_\varphi$, i.e. $R_\varphi:f \mapsto f \circ \varphi$, we get from the above that
\[
	R_\varphi:H^\sigma(\R) \to H^\sigma(\R),\;f \mapsto f \circ \varphi,\quad 0 \leq \sigma \leq s+1,
\]
is a continuous linear map. For $1/2 < s < 1$ composition from the right $R_\varphi:H^\sigma(\R) \to H^\sigma(\R),\;s-1 \leq \sigma < 0$, is a well defined continuous linear map -- see \cite{gn_lagrangian}. Note that $R_\varphi$ is an isomorphism with inverse $R_\varphi^{-1}=R_{\varphi^{-1}}$.\\
From the above we get that for $\varphi \in \Ds^{s+1}(\R)$ multiplication with $\varphi_x$
\[
	M_{\varphi_x}:H^\sigma(\R) \to H^\sigma(\R),\;f \mapsto \varphi_x \cdot f,\quad \min\{s-1,0\} \leq \sigma \leq s,
\]
is a continuous linear map. From \cite{composition} we know that division by $\varphi_x$ 
\[
	M_{\varphi_x}^{-1}:H^\sigma(\R) \to H^\sigma(\R),\; f \mapsto \frac{f}{\varphi_x},\quad \min\{s-1,0\} \leq \sigma \leq s,
\]
is continuous as well. The map
\[
	\Ds^{s+1}(\R) \to L(H^\sigma(\R);H^\sigma(\R)),\;\varphi \mapsto M_{\varphi_x},\quad \min\{s-1,0\} \leq \sigma \leq s,
\]
is affine linear, hence it is analytic -- see \cite{lagrangian} for basic definitions and results on analyticity in Banach spaces. Here $L(X;Y)$ is the space of continuous linear maps from $X$ to $Y$. Using Neumann series we see that inversion of isomorphisms is an analytic process. So
\[
	\Ds^{s+1}(\R) \to L(H^\sigma(\R);H^\sigma(\R)),\;\varphi \mapsto M_{\varphi_x}^{-1},\quad \min\{s-1,0\} \leq \sigma \leq s,
\]
is analytic as well. In particular
\begin{align}\label{det}
	\Ds^{s+1}(\R) \times (1+H^s(\R)) \to 1+H^s(\R),\; (\varphi,f) \mapsto M_{\varphi_x}^{-1} f = \frac{f}{\varphi_x},
\end{align}
is analytic.\\
Finally consider 
\[
	[(f \circ \varphi^{-1})_x] \circ \varphi= \frac{f_x}{\varphi_x}.
\]
In other words we have $R_\varphi \partial_x R_\varphi^{-1}=M_{\varphi_x}^{-1} \partial_x$. Thus
\begin{align}\label{conj}
	\Ds^{s+1}(\R) \to L(H^\sigma(\R);H^{\sigma-1}(\R)),\;\varphi \mapsto R_\varphi \partial_x R_\varphi^{-1},\quad \sigma=s,s+1,
\end{align}
is analytic.\\
Suppose now that for some $T > 0$ we have $u \in C([0,T];H^{s+1}(\R))$. In \cite{lagrangian} it was shown that there is a unique $\varphi \in C^1([0,T];\Ds^{s+1}(\R))$ satisfying
\[
	\varphi_t(t)=u(t) \circ \varphi(t),\;0 \leq t \leq T,\;\varphi(0)=\text{id}.
\]
Thus $\Ds^{s+1}(\R)$ is the right functional space for the Lagrangian variable $\varphi$.\\
For the statement of the main result of this section we introduce for $s > 1/2$
\[
	U^s=\{h \in 1+H^s(\R) \;|\; \inf_{x \in \R} h(x) > 0\}.
\]
Note that we can identify $U^s$ with the open subset $U^s-1 \subset H^s(\R)$.

\begin{Lemma}\label{lemma_analytic}
	Let $s > 1/2$. Then
	\[
		\Ds^{s+1}(\R) \times H^{s+1}(\R) \times U^s \times H^{s+3}(\R) \to H^{s+1}(\R),\;(\varphi,v,h_0,\xi) \mapsto F(\varphi,v,h_0,\xi), 
	\]
is $C^1$. Here $F$ is the map from \eqref{F}.
\end{Lemma}

\begin{proof}
We write
	\[
		F(\varphi,v,h_0,\xi)=-R_\varphi A_{(h_0/\varphi_x)\circ \varphi^{-1},\xi}^{-1} R_\varphi^{-1} R_\varphi P((h_0/\varphi_x) \circ \varphi^{-1},v \circ \varphi^{-1},\xi),
	\]
with $P$ from \eqref{nonlocal}. We will first show that
	\begin{align}
		\label{inv}
		\begin{split}
		&\Ds^{s+1}(\R) \times U^s \times H^{s+3}(\R) \to L(H^{s-1}(\R);H^{s+1}(\R)),\\
		&(\varphi,h_0) \mapsto R_\varphi A_{(h_0/\varphi_x)\circ \varphi^{-1},\xi}^{-1} R_\varphi^{-1},
		\end{split}
	\end{align}
is $C^1$. Note that
	\[
		R_\varphi A_{(h_0/\varphi_x)\circ \varphi^{-1},\xi}^{-1} R_\varphi^{-1}=\left(R_\varphi A_{(h_0/\varphi_x)\circ \varphi^{-1}} R_\varphi^{-1}\right)^{-1}.
		\]
Let $f \in H^{s+1}(\R)$. We then have
	\begin{align*}
		&R_\varphi A_{(h_0/\varphi_x)\circ \varphi^{-1},\xi} R_\varphi^{-1} f =\\ 
		&M_{\varphi_x}^{-1} h_0 \cdot (1+(R_\varphi \xi_x)^2) \cdot f + R_\varphi \partial_x R_\varphi^{-1} [\frac{1}{2} (M_{\varphi_x}^{-1} h_0)^2 \cdot (R_\varphi \xi_x) \cdot f]\\
		&-\frac{1}{2} (M_{\varphi_x}^{-1} h_0)^2 \cdot R_\varphi \xi_x \cdot R_\varphi \partial_x R_\varphi^{-1} f - R_\varphi \partial_x R_\varphi^{-1} [\frac{1}{3} (M_{\varphi_x}^{-1}h_0)^3 \cdot R_\varphi \partial_x R_\varphi^{-1} f].
	\end{align*}
From \eqref{det} resp. \eqref{conj} we know that $M_{\varphi_x}h_0$ resp. $R_\varphi \partial_x R_\varphi^{-1}$ depend analytically on $\varphi$ and $h_0$. From \eqref{C1} we know that $R_\varphi \xi_x$ depends in a $C^1$ fashion on $\varphi$ and $\xi$. Thus
	\begin{align*}
		&\Ds^{s+1}(\R) \times U^s \times H^{s+3}(\R) \to L(H^{s+1}(\R);H^{s-1}(\R)),\\
		&(\varphi,h_0) \mapsto R_\varphi A_{(h_0/\varphi_x)\circ \varphi^{-1},\xi} R_\varphi^{-1},
	\end{align*}
	is $C^1$. By using Neumann series we know that inversion of isomorphisms is a smooth process. Hence we get that the map in \eqref{inv} is $C^1$. Next consider
\begin{align*}
	&R_\varphi P((h_0/\varphi_x) \circ \varphi^{-1},v \circ \varphi^{-1},\xi)=\\
	&M_{\varphi_x}^{-1} h_0 \cdot R_\varphi \partial_x R_\varphi^{-1} M_{\varphi_x}^{-1}h_0 \cdot g + R_\varphi \partial_x R_\varphi^{-1} [\frac{1}{2} (M_{\varphi_x}^{-1} h_0)^2 \cdot v^2 \cdot R_\varphi \xi_{xx}]\\
	&+R_\varphi \partial_x R_\varphi^{-1} [\frac{2}{3} (M_{\varphi_x}^{-1} h_0)^3 (R_\varphi \partial_x R_\varphi^{-1} v)^2] + M_{\varphi_x}h_0 \cdot R_\varphi \xi_x \cdot g \\
	&+M_{\varphi_x}^{-1} h_0 \cdot v^2 \cdot R_\varphi \xi_x \cdot R_\varphi \xi_{xx} + (M_{\varphi_x}^{-1}h_0)^2 \cdot R_\varphi \xi_x \cdot (R_\varphi \partial_x R_\varphi^{-1}v)^2.
\end{align*}
	As above we have that $M_{\varphi_x}^{-1} h_0$ depends analytically on $(\varphi,h_0)$, $R_\varphi \partial_x R_\varphi^{-1}$ depends analytically on $\varphi$ and $R_\varphi \xi_x$ resp. $R_\varphi \xi_{xx}$ depend in a $C^1$ way on $(\varphi,\xi)$. Altogether we see that the map in the statement of Lemma \ref{lemma_analytic} is $C^1$. This finishes the proof.
\end{proof}

Consider now for $(h_0,u_0,\xi) \in U^s \times H^{s+1}(\R) \times H^{s+3}(\R)$ the second order ODE on $\Ds^{s+1}(\R)$
\begin{align}\label{ode}
	\varphi_{tt}=F(\varphi,\varphi_t,h_0,\xi),\;\varphi(0)=\text{id},\;\varphi_t(0)=u_0.
\end{align}
By Lemma \ref{lemma_analytic} this is a $C^1$ ODE. Hence we can find solutions by applying the Picard-Lindel\"of Theorem. The equation \eqref{ode} is a Lagrangian formulation of \eqref{gn}.

\section{Local Well-Posedness of the Green-Naghdi system}\label{section_lwp}

The goal of this section is to prove Theorem \ref{th_lwp}.

\begin{proof}[Proof of Theorem \ref{th_lwp}]
	Let $s > 1/2$ and $(h_0,u_0,\xi) \in U^s \times H^{s+1}(\R) \times H^{s+3}(\R)$. First we show the existence of a solution. By the Picard-Lindel\"of Theorem there is some $T > 0$ and a solution $\varphi \in C^2([0,T];\Ds^{s+1}(\R))$ to \eqref{ode}. We define 
\[
	h(t):=\left(\frac{h_0}{\varphi_x(t)}\right) \circ \varphi(t)^{-1},\;u(t)=\varphi_t(t) \circ \varphi(t)^{-1},\quad 0 \leq t \leq T.
\]
	By the continuity properties of the composition \eqref{composition} and inversion \eqref{inversion} we get
\[
	(h,u) \in C([0,T];U^s \times H^{s+1}(\R)).
\]
	We also have by a general argument (see \cite{gn_lagrangian}) 
	\[
		(h,u) \in C^1([0,T];(1+H^{s-1}(\R)) \times H^s(\R)).
	\]
	By the Sobolev imbedding we get $u \in C^1([0,T] \times \R)$. Taking pointwise the $t$ derivative gives
\[
	\frac{d}{dt} u \circ \varphi = (u_t+uu_x) \circ \varphi = \varphi_{tt}=-\left(A_{h,\xi}^{-1} P(h,u,\xi)\right) \circ \varphi.
\]
Thus pointwise 
	\[
		u_t + uu_x = -A_{h,\xi}^{-1} P (h,u,\xi).
	\]
	But this is an identity in $H^s(\R)$ as well. So $(h,u)$ solves the second equation in \eqref{gn}. By the very definition of $h$ we see that $(h,u)$ solves the first equation in \eqref{gn}. So the existence of solutions is established.\\
	In Lagrangian variables we get continuous dependence on $(h_0,u_0,\xi)$ by ODE theory. The continuity properties of the composition \eqref{composition} and inversion \eqref{inversion} now provide continuous dependence of $(h,u)$ on $(h_0,u_0,\xi)$. This proves continuous dependence on the data.\\
	To prove uniqueness we assume that we have two solutions
\[
	(h,u),(\tilde h,\tilde u) \in C([0,T];U^s \times H^{s+1}(\R)) \times C^1([0,T];(1+H^{s-1}(\R)) \times H^s(\R))
\]
	satisfying \eqref{gn}. We know that $u$ resp. $\tilde u$ generate flow maps $\varphi$ resp. $\tilde \varphi$ with
\[
	\varphi,\;\tilde \varphi \in C^1([0,T];\Ds^{s+1}(\R)). 
\]
	As a consequence of $h_t+(hu)=0$ we get
	\[
		h=\left(\frac{h_0}{\varphi_x}\right) \circ \varphi^{-1}.
		\]
Taking pointwise the $t$ derivative in the equation $\varphi_t = u \circ \varphi$ we get
	\[
		\varphi_{tt}=(u_t+uu_x) \circ \varphi.
	\]
Thus we get pointwise
	\[
		\varphi_{tt}=-R_\varphi A_{h,\xi}^{-1} P(h,u,\xi)=F(\varphi,\varphi_t,h_0,\xi),
\]
	where we used $h=(h_0/\varphi_x) \circ \varphi^{-1}$. This is an identity in $H^{s+1}(\R)$ as well. So $\varphi$ solves the ODE \eqref{ode}. Similarly we get that $\tilde \varphi$ solves \eqref{ode} with the same initial data. By the Uniqueness Theorem for ODEs we get $\varphi \equiv \tilde \varphi$ on $[0,T]$ and with that $(h,u)\equiv(\tilde h,\tilde u)$ on $[0,T]$.\\
Putting existence, continuous dependence on the data and uniqueness together we get the local well-posedness of \eqref{gn}. This finishes the proof. 
\end{proof}

\bibliographystyle{plain}

\flushleft
\author{ Hasan \.{I}nci\\
Department of Mathematics, Ko\c{c} University\\
Rumelifeneri Yolu\\
34450 Sar{\i}yer \.{I}stanbul T\"urkiye\\
        {\it email: } {hinci@ku.edu.tr}
}

\end{document}